\documentclass[11pt,preprint,reqno]{amsart} 
\usepackage{amsmath,amssymb,amsthm}
\usepackage{mathscinet} 
\usepackage[usenames,dvipsnames]{xcolor}
\usepackage[mathscr]{eucal}

\usepackage{syntonly}  

\usepackage{dsfont} 
\usepackage{graphicx}
\usepackage{xcolor}
\newtheorem{theorem}{Theorem}

\newtheorem{proposition}[theorem]{Proposition}
\newtheorem{definition}[theorem]{Definition}
\newtheorem{corollary}[theorem]{Corollary}
\newtheorem{lemma}[theorem]{Lemma}

\newtheorem{remark}{Remark}

\def\qed{\hbox{${\vcenter{\vbox{		 
   \hrule height 0.4pt\hbox{\vrule width 0.4pt height 6pt
   \kern5pt\vrule width 0.4pt}\hrule height 0.4pt}}}$}}

\def\cC{\mathcal C}

\def\cE{\mathcal E}
\def\cF{\mathcal F}

\def\cH{\mathcal H}

\def\cS{\mathcal S}

\def\bD{\mathbb D}
\def\bE{\mathbb E}

\def\bN{\mathbb N}

\def\bP{\mathbb P}

\def\bR{\mathbb R}

\def\uh{\underline{h}}
\def\uu{\underline{u}}
\def\uz{\underline{z}}

\def\veps{\varepsilon}

\def\uno{\mathsf{1}}

\def\uu{\underline u}

\newcommand{\comment}[1]{$\null$}

\begin{document}

\title[Absolute continuity of the law for SPDEs with boundary noise]{Absolute continuity of the law
for solutions of stochastic differential equations with boundary noise}

\author
{Stefano Bonaccorsi}
\author{Margherita Zanella}

\address{Stefano Bonaccorsi \newline
Dipartimento di Matematica, Universit\`a di
    Trento, via Sommarive 14, 38123 Povo (Trento), Italia 
}
\email{stefano.bonaccorsi@unitn.it}
\address{Margherita Zanella\newline
Dipartimento di Matematica, Universit\`a di Pavia, via Ferrata 5, 27100 Pavia, Italia}
\email{margherita.zanella01@ateneopv.it}


\subjclass[2000]{60H15, 60H07}
\keywords{Fractional Brownian motion, stochastic boundary condition, density of the solution.}

\begin{abstract}
We study existence and regularity of the density for the solution $u(t,x)$ (with fixed $t > 0$ and $x \in D$)
of the heat equation in a bounded domain $D \subset \bR^d$
driven by a stochastic inhomogeneous Neumann boundary condition with stochastic term. 
The stochastic perturbation is given by a fractional Brownian motion process. 
Under suitable regularity assumptions on the coefficients, by means of tools from the Malliavin calculus, 
we prove that the law of the solution has a smooth density with respect to the Lebesgue measure in $\bR$.
\end{abstract}

\maketitle

\def\uh{\underline{h}}
\def\uu{\underline{u}}
\def\uz{\underline{z}}


\section{Introduction}

Let $D$ be a bounded convex domain in $\bR^d$ with smooth $C^\infty$ boundary $\partial D$,
and denote $\nu$ the inward-pointing normal vector to $\partial D$.
In the domain $D$, let us consider a heat equation with inhomogeneous boundary conditions
of Neumann type 
\begin{align}
\label{e1}
\begin{cases}
\displaystyle \partial_t u(t,x) - \tfrac12 \Delta u(t,x) = 0,\qquad & t > 0,\ x \in D
\\
\langle \nu, \nabla u(t,\bar x) \rangle + \beta u(t,\bar x) + g(t,\bar x) = 0 \qquad & \bar{x} \in \partial D
\\
u(0,x) = 0 & x \in D
\end{cases}
\end{align}
where $g$ is given and
$\beta$ is a positive constant. 

Evolution problems of this kind have been studied by several authors; in particular,
if $g$ were a deterministic function then equation \eqref{e1} 
becomes a special case of the general theory of inhomogeneous boundary value problem, that is well studied in the literature, see for instance \cite{Lions1972}.
In this case, a solution for the heat problem is given by 
\begin{align}\label{e2}
u(t,x) = \int_0^t \int_{\partial D} p_N(s,x,\bar y) g(s,\bar y) \, {\rm d}s,
\end{align}
where $p_N(t, x, \bar{y})$ is the Poisson kernel of problem \eqref{e1} (for further details see Section \ref{studio_kernel}).

In our knowledge, the first extension to the case of stochastic boundary condition is due to \cite{Sowers1994} 
in case of an additive boundary white noise $g(t,\bar x) = \alpha(\bar x) \frac{{\rm d}}{{\rm d}t}B(t)$ with respect to a Brownian motion process $B$.
A semigroup approach to the same equation is given in \cite{DaPrato1992}; the wave equation with stochastic boundary values
was investigated in \cite{Mao1993}.
\\
Stochastic perturbation driven by fractional Gaussian processes have been studied recently by several authors, see
\cite{Duncan2002, Duncan2005, Maslowski2003, Bonaccorsi2011b} 
especially in the case of equations perturbed by distributed noise.
Less results are available for the case of equations perturbed by a boundary noise: see \cite{Duncan2002, Bonaccorsi2010a, Barbu2015a, Bonaccorsi2016}. In this last paper, authors allow also a nonlinear term on the boundary (both for the Brownian motion and the fractional Brownian motion).

Here, we aim to investigate further the properties of the solution
for the problem introduced in \cite{Barbu2014}.
We consider the case when the boundary condition are given by the sum of a {\em nonlinear} function of the solution
and a stochastic perturbation:
\begin{align*}
g(t,\bar x) = g(u(t,\bar x)) - \int_S \alpha(\sigma, \bar x) \, {\rm d} B(\sigma, t),
\end{align*}
where $B(t,\sigma)$ is  a fractional Brownian motion of Hurst parameter $H \in (\frac12,1)$ on $[0,T] \times S$
and we assume that $g: \bR \to \bR$ is a Lipschitz continuous function. 
\\
Thus we write the problem at hand in the form
\begin{equation}
\label{eq:nl}
\begin{cases}
\partial_t u(t,x) - \frac12 \Delta u(t,x) =0, \qquad & t > 0,\ x \in D,
\\ 
\langle \nu, \nabla u(t,\bar x) \rangle + \beta u(t,\bar x) + g(u(t,\bar x))= \frac{{\rm d}}{{\rm d}t} \int_S \alpha(\sigma, \bar x) \, {\rm d} B(\sigma, t), & \bar{x} \in \partial D,
\\ 
u(0,x) = 0 & x \in D.
\end{cases}
\end{equation}

First, we prove the existence, uniqueness and 
continuity of the solution
$u(t,x)$ with respect to both parameters $t$ and $x$.
Further, we show, using tools of the Malliavin calculus,
that the random variable $u(t,x)$ has an absolutely continuous probability law and, under suitable assumptions,
the density of this random variable is smooth.

\smallskip

We interpret equation \eqref{eq:nl} in the sense of Walsh \cite{Walsh1986}: a solution of \eqref{eq:nl}
is the process $u(t,x)$ that satisfies the evolution equation
\begin{equation}
\label{eq:nl2}
u(t,x) = \int_0^t \int_{\partial D} p_N(t-s,x,\bar y) g(u(s,\bar y)) \, {\rm d}\bar y \, {\rm d}s
+ \int_0^t \int_S \int_{\partial D} p_N(t-s,x,\bar y) \alpha(\sigma, \bar y) \,  {\rm d}\bar y \,{\rm d}B(\sigma,s).
\end{equation}

In order to give a sense to the random field formulation for the solution we have to evaluate the process in both time and space. It is then necessary to prove that there exists a continuous version (in both the entries) of the process. As regards the continuity in time, it is a classical result and it follows immediately from \eqref{eq:nl2}. As regards the space regularity, in the following we prove that we can consider a continuous version of the solution process \textit{inside} the domain $D$; unfortunately this kind of regularity can not be extended to the boundary $\partial D$ as well, where we have only a regularity of $L^p(\partial D)$ type.
Our main aim is to prove the existence of a smooth density for the solution at time $t$ and space $x\in D$ fixed. In order to obtain this kind of result we need, as a necessary step, to study the solution process, as well as its Malliavin derivative, on the boundary. We point out here that, since we do not have the continuity in space on the boundary, all the results stated for $u(t, \xi)$ with $\xi \in \partial D$ (and for its Malliavin derivative) are valid \textit{a.e.} and, with an abuse of notation, by $u(t,\xi)$ we mean a representative element in the equivalence class $L^p(\partial D)$. 
Nevertheless, as regards the analogous results stated for the points inside the domain, they are valid everywhere thanks to the representation formula for the solution and its Malliavin derivative (see \eqref{sol_int} and \eqref{der_sol_int}).

Let us state the main assumptions that we impose on the coefficients of the problem.

\begin{description}
\item[(g1)] {\em Lipschitz continuity.} We assume that $g \in C^1(\bR)$ is Lipschitz continuous, with
$|\partial_ug(u)|\le L$ and $g(u) \le L(1+|u|)$ for some constant $L > 0$.
\item[(g2)] {\em Smoothness.} Further to {\bf (g1)}, in order to prove the smoothness of the law of the solution $u(t,x)$
we shall require that $g \in C^\infty(\bR)$ and 
derivatives of all orders are bounded by a (finite) constant $L$:
$\|\partial_u^ng(u)\|\le L$ for every $n \ge 1$.
\item[(a1)] {\em Integrability of the boundary coefficient.} We  assume that $\alpha \in L^2(S \times \partial D)$.
\item[(a1')] {\em Further integrability of the boundary coefficient.} In addition to {\bf (a1)},
for some results we will require a stronger integrability condition of the boundary coefficient, namely that 
$\alpha \in L^{\theta}(\partial D; L^2(S))$ for $\theta >\frac{d-1}{2H-1}$.
\item[(a2)] {\em Non-degeneracy.} There exists $\alpha_0>0$ such that $\alpha(\sigma, x) \ge \alpha_0$.
\end{description}

This paper is in principle an extension of the results in \cite{Barbu2014, Barbu2015a, Bonaccorsi2016}.
Section \ref{studio_kernel} contains some preliminaries, basic definitions and construction of the
heat kernel in a bounded domain. Although these results should be classical, we have find it useful
to collect them in a unique place; moreover, we did not find a reference for some of the estimates we
need here. 
\\
Section \ref{sec2} contains a review of some basic facts about Malliavin calculus for fractional Gaussian processes
(our approach here is based mainly on \cite{Nualart2006}). 
\\
Section \ref{scp} is devoted to the analysis of the stochastic convolution process: we first study the regularity of the process and then we compute its Malliavin derivative. We also prove that its law admits a density with respect to the Lebesgue measure. 
\\
In Section \ref{Eos} we recall the results concerning the existence and regularity of the solution of \eqref{eq:nl} proved in \cite{Barbu2014} and we study the Malliavin derivative of the solution. Finally in Sections \ref{Edsnhe} 
 we study the regularity of the solution in the Malliavin sense.

\smallskip

{\small Notation. In the sequel, we shall indicate with $C$ a constant that may vary from line to line. 
In certain cases, we write $C_{\alpha,\beta,\dots}$ to 
emphasize the dependence of the constant on the parameters $\alpha, \beta, \dots$.}



\section{Some useful estimates on the Poisson kernel}
\label{studio_kernel}

In this section we review the construction of the fundamental solution of the heat equation with inhomogeneous boundary
conditions and we state the relevant estimates that we need in the sequel. There is a large literature concerning this
subject: we shall refer, for instance, to \cite[Chapter 5]{Friedman1964}.
\\
Let $D \subset \bR^d$ be a bounded domain, with smooth boundary $\partial D$. We denote $\sigma({\rm d}y)$ the 
surface measure on $\partial D$ and we assume that $|\partial D| < \infty$.
\\
Let us consider the problem 
\begin{equation}
\label{det_heat_eq}
\begin{cases}
\frac{\partial }{\partial t}u(t,x)- \frac 12 \Delta u(t,x)=0 \,  \qquad & t \in (0,T), \quad x \in \mathring{D}
\\
u(0,x) = 0 & x \in D
\\
\tfrac{\partial}{\partial \nu}u(t,\xi) +  \beta u(t,\xi) = g(t, \xi) & t \in (0,T), \quad \xi \in \partial D
\end{cases}
\end{equation}
Set $L=\frac 12 \Delta- \frac{\partial}{\partial t}$. 
We extend the coefficients of $L$ to a larger set $\Omega_0:=D_0\times \left[0,T\right]$, where $D_0$ is a bounded domain 
containing $\bar{D}$ (notice that the extension is trivial since the coefficients are constant).
The fundamental solution $\Gamma(t, x, s,y)$ can be constructed in $\Omega_0$. Let us recall that

\begin{definition}
A fundamental solution of $Lu=0$ in $\Omega_0$ is a function $\Gamma$ defined for all $(t,x)\in \Omega_0$, $(s,y) \in \Omega_0$,
$t>s$, which satisfies the following conditions:
\begin{enumerate}
\item for fixed $(s,y)\in \Omega_0$, it satisfies the equation 
$Lu=0$; 
\item for any continuous function $f(x) \in C(\bar{D_0})$, if $x \in D_0$ then
$$\lim_{t \downarrow s}\int_{D_0}\Gamma(t,x,s,y)f(y){\rm d}y =f(x).$$
\end{enumerate}
\end{definition}

According to
\cite[Chapter 5.1]{Friedman1964}, the fundamental solution of the homogeneous problem associated to \eqref{det_heat_eq} 
is  given by a function $Z(t,x,s,y)$, that is the heat kernel for the $d$-dimensional Brownian motion associated 
with the diffusion operator $L$, so it can be written as
\begin{equation}
Z(t,x,s,y)=(2\pi(t-s))^{-\frac d2}\exp{\left(-\frac{|x-y|^2}{2t}\right)}.
\end{equation}
According to this representation, the following inequalities hold for $x,y \in D_0$ and $0\le s<t\le T$:
\begin{align}
\label{stima_kernel_hom}
|\Gamma(t,x,s,y) | &\le c (t-s)^{-\mu}|x-y|^{2 \mu-d}, \quad 0<\mu<1;
\\%
\label{stima_der_kernel_hom}
|\nabla_x\Gamma(t,x,s,y) | &\le c (t-s)^{-\mu}|x-y|^{2 \mu-d-1}, \quad \frac12<\mu<1.
\end{align}

Let us now return to problem \eqref{det_heat_eq}. 
The following result (see \cite[Theorem 5.3.2]{Friedman1964}) provides an explicit formula 
for the associated kernel and a first useful estimate.

\begin{theorem}
\label{stima}
If $g$ is a continuous function on $\partial D \times \left[0,T\right]$, then there exists a unique solution of the problem \eqref{det_heat_eq}, and it has the form
\begin{equation}
u(t,x)= \int_0^t \int_{\partial D}p_N(t,x,s,y)g(s,y)\sigma({\rm d}y) {\rm d}s,
\end{equation}
where 
\begin{equation}
\label{kernel_expres}
p_N(t,x,s,y)=-2\Gamma(t,x,s,y)-2 \sum_{n=1}^{\infty}\int_s^t\int_{\partial D}\Gamma(t,x,s,y)M_n(r,z,s,y)\sigma({\rm d}z){\rm d}r
\end{equation}
and the terms $M_n$ are defined recursively as follows
\begin{align*}
M_1(t,x,s,y)& := M(t,x,s,y) = \left[\frac{\partial \Gamma}{\partial \nu(x)}+ \beta\Gamma\right](t,x,s,y);
\\
M_{n+1}& := \int_0^t\int_{\partial D}M(t,x,r,z)M_n(r,z,s,y) \sigma({\rm d}y) {\rm d}s.
\end{align*}
\end{theorem}

We shall not provide the proof of the theorem; however, we concentrate on the following consequence of the construction.
Some technical lemmas, that are needed in order to get the proof, are given in the Appendix.

\begin{corollary}
\label{co:stima}
For every $(t,x) \in \left[0,T\right] \times \bar{D}$, the following estimate holds
\begin{equation}
\label{stima_ok}
p_N(t,x,s,y) \le c(t-s)^{-\mu}|x-y|^{2\mu-d}, \qquad {\frac12<\mu<1}.
\end{equation}
\end{corollary}

\begin{proof}
We start from \eqref{stima_kernel_hom}; using Lemma \ref{stim_kernel_bordo} we have
\begin{align*}
|M(t,x,s,y)| \le c(t-s)^{-\mu}|x-y|^{2 \mu-d}.
\end{align*}
So, proceeding by induction, and using Lemma  \ref{integral}, we can give an explicit estimate for all the $M_n$'s. 
For $M_1$ we have the estimate above. Let us explicitly show the case $M_2$.
\begin{align*}
|M_2(t,x,s,y)| &\le c^2\int_s^t\int_{\partial D}(t-r)^{-\mu}|x-z|^{2 \mu-d}(r-s)^{-\mu}|z-y|^{2 \mu-d} \sigma({\rm d} z) {\rm d}r\\
&= c^2\left(\int_s^t(t-r)^{-\mu}(r-s)^{- \mu}{\rm d}r\right)|x-y|^{4 \mu-d-1}\\
&\le c^2 \frac{\Gamma(1-\mu)\Gamma(1-\mu)}{\Gamma(2-2\mu)}(t-s)^{1-2 \mu}|x-y|^{4\mu-d-1}
\end{align*}
Then the induction step allows us to conclude
$$|M_k(t,x,s,y)| \le \gamma_k(\mu)c^k(t-s)^{k(1-\mu)-1}|x-y|^{2k\mu-d-k+1}$$
where
$$\gamma_k(\mu)=\prod_{j=1}^kC_j(\mu)$$
with
$$C_j(\mu)=\frac{\Gamma(1-\mu)\Gamma(j(1-\mu)}{\Gamma((j+1)(1-\mu))}.$$
It follows immediately that
\begin{equation}
|M_k(t,x,s,y)| \le c^k\frac{(\Gamma(1-\mu))^{k+1}}{\Gamma((k+1)(1-\mu))}(t-s)^{k(1-\mu)-1}|x-y|^{k(2\mu-1)+1-d}.
\end{equation}

Using previous estimate on the $M_k$'s and \eqref{stima_kernel_hom}, we can now compute the desired estimate of the kernel, 
starting from expression \eqref{kernel_expres}.
Let us start by estimating an arbitrary term of the series in \eqref{kernel_expres} (we need again Lemma \ref{integral}):
\begin{align*}
\int_s^t\int_{\partial D} & \Gamma(t,x,s,y) M_n(r,z,s,y)\sigma({\rm d}z){\rm d}r\\
&\le c^k\frac{(\Gamma(1-\mu))^{k+1}}{\Gamma((k+1)(1-\mu))}\int_s^t\int_{\partial D} (t-r)^{-\mu}|x-y|^{2 \mu-d}(r-s)^{k(1-\mu)-1}|y-z|^{k(2\mu-1)+1-d}\sigma({\rm d}z){\rm d}r\\
&\le c^k\frac{(\Gamma(1-\mu))^{k+1}}{\Gamma((k+1)(1-\mu))} \frac{\Gamma(1-\mu)\Gamma(k(1-\mu))}{\Gamma((k+1)(1-\mu))} (t-s)^{k (1-\mu)-\mu} C |x-y|^{k(2\mu-1)+2\mu-d}
\end{align*}
and so we get
\begin{multline*}
|p_N(t,x,sy)| \le c(t-s)^{-\mu}|x-y|^{2 \mu-d}
\\
+ \left(c \sum_{k=1}^{\infty}\frac{c^{k+1}(\Gamma(1-\mu))^{k+1}}{\Gamma((k+1)(1-\mu))}(t-s)^{k(1-\mu)}|x-y|^{k(2 \mu-1)}\right)|x-y|^{2 \mu-d} (t-s)^{-\mu},
\end{multline*}
where the series is convergent provided $\frac12<\mu <1$. So we recover the desired estimate
\begin{equation}
p_N(t,x,s,y) \le c(t-s)^{-\mu}|x-y|^{2\mu-d}.
\end{equation}
\end{proof}

\begin{remark}
Notice that the estimate made on the kernel $p_N(t,x,s,y)$ is valid also when $ x \in \partial D$ thanks to Lemma  \ref{stim_kernel_bordo} and the fact that formula \eqref{stima_kernel_hom} is true for all $x \in D_0 \supset \bar{D}$.
\end{remark}








We state now a lower bound estimates for the Poisson kernel: we refer to \cite[Theorem 4.1]{Yang2013}.

\begin{lemma}
\label{lower}
Suppose that domain $D$ is convex. There exist constants $C_1$, $C_2>0$ and $T>0$, such that for $x,y \in D, t \in \left[0,T\right]$
\begin{equation}
\label{lower_bound}
p_N(t,x,y) \ge C_1t^{-\frac d2} \emph{exp}\left(-\frac{C_2|x-y|^2}{t}\right).
\end{equation}
\end{lemma}

Finally we give some useful estimates about the spatial derivative of the kernel.
The following lemma is provided in \cite[Proposition 3.3]{Sowers1994}

\begin{lemma}\label{l.app1}
There exists a constant $k \le 1$ such that
\begin{equation}
\label{stima_der_kernel}
|\nabla p_N(t,x,\bar y)| \le k^{-1} \exp(-k |x-\bar y|^2/t) \, t^{-(d+1)/2}
\end{equation}
for all $t \in (0,T)$, $x \in \mathring{D}$ and $\bar y \in \partial D$.
\end{lemma}

The above estimate is not so easy to deal with, because of the presence of the Gaussian exponential. 
However, if we recall the  inequalities \eqref{barbu:e1} and \eqref{c1}, i.e.,
\begin{equation*}
t^{-n/2} e^{-k r^2/t} \le K t^{-\mu} r^{-n+2\mu}
\end{equation*}
then we have the following.

\begin{corollary}
\label{c.app1}
There exists a constant $k > 0$ such that
\begin{equation}
\label{stima2_der_kernel}
|\nabla p_N(t,x,\bar y)| \le k \, |x-\bar y|^{-d+2\mu} \, t^{-(2\mu+1)/2}
\end{equation}
for all $t \in (0,T)$, $x \in \mathring{D}$ and $\bar y \in \partial D$.
\end{corollary}


\section{Preliminaries on Malliavin calculus}
\label{sec2}

Let us recall some basic facts about the Malliavin calculus with respect to (standard and fractional) Brownian motion; for full details,
we refer to \cite{Nualart2006}.

Fix a measurable space $(S,\cS)$ with a finite measure $\mu$ on it, as well as a time interval $[0, T]$.

We are given a complete probability space $(\Omega,\cF,\bP)$ and a centered Gaussian family $B = \{B(h),\, h \in \cH\}$ defined in $\Omega$.
The space $\cH$ is constructed below.

Recall that a fractional Brownian motion $B^H = \{B^H(t),\ t \in [0,T]\}$
is a centered Gaussian process with covariance function
\begin{align}
\label{cov}
R_H(t,s) = \frac12 (s^{2H} + t^{2H} - |t-s|^{2H}), \qquad s, t \in [0,T].
\end{align}

Let $\cE$ be the space of step functions on $[0,T] \times S$. 
We denote by $\cH$ the closure of $\cE$ with 
respect to the scalar product
\begin{align*}
\langle {\mathds 1}_{[0,t]} \times {\mathds 1}_A, {\mathds 1}_{[0,s]} \times {\mathds 1}_B \rangle_{\cH} = R_H(t,s) \,  \mu(A \cap B);
\end{align*}
notice that in case $H=1/2$ then the first component in $\cH$ is the standard $L^2$ space with respect to the Lebesgue measure on $[0,T]$, so that
for  $\phi, \psi \in \cH$ we have
\begin{align*}
\langle \phi, \psi \rangle_{\cH} = \int_0^T \int_S \phi(s,\sigma) \psi(s,\sigma) \, \mu({\rm d}\sigma) \, {\rm d}s.
\end{align*}
In case of a fractional Brownian motion with Hurst parameter $H > 1/2$ it holds
\begin{align*}
\langle \phi, \psi \rangle_{\cH} =  \int_S \int_0^T \int_0^T |s-t|^{2H-2} \phi(s,\sigma) \psi(t,\sigma) \, {\rm d}t \, {\rm d}s  \, \mu({\rm d}\sigma),
\end{align*}
$\phi, \psi \in \cH$.

Thus, in case $H = 1/2$, we say that the Gaussian family $B$ is associated to a Brownian motion process
$B_{s,\sigma}$ on $\cH$ and in case $H > \frac12$ it is associated to a fractional Brownian motion $B_{s,\sigma}$
via the identification
\begin{align*}
B(\phi) =  \int_0^T \int_S \phi(s,\sigma) \, {\rm d}B_{s,\sigma}, \qquad \phi \in \cH.
\end{align*}


A $\cF$-measurable real
valued random variable $F$ is said to be cylindrical if it can be
written as
\begin{equation*}
F=f \left( B(\phi^1) ,\ldots, B(\phi^n) \right)\;,
\end{equation*}
where $\phi^i \in \cH$ and $f:\bR^n \to \bR$ is a $C^{\infty}$ bounded function. The set of
cylindrical random variables is denoted $\mathcal{S}$. The
Malliavin derivative of $F \in \mathcal{S}$ is the 
stochastic process $D F = \{D_{s,\sigma} F,\ s \in [0,T],\ \sigma \in S\}$ given by
\[
{D}_{s,\sigma} F=\sum_{i=1}^{n} \phi^i (s,\sigma) \frac{\partial f}{\partial
x_i} \left( B(\phi^1) ,\ldots, B(\phi^n) \right).
\]
More generally, we can introduce iterated derivatives. If $F \in
\mathcal{S}$, we set
\[
D^k_{t_1,\ldots,t_k, \sigma_1, \ldots, \sigma_k} F = D_{t_1,\sigma_1}
\ldots D_{t_k,\sigma_k} F.
\]
For any $p \geq 1$, the operator $D^k$ is closable from
$\mathcal{S}$ into $L^p \left( \cC ( [0,T] \times S ,
\bR) , \cH^{\otimes k} \right)$. We denote by
$\bD^{k,p}(\mathcal{H})$ the closure of the class of
cylindrical random variables with respect to the norm
\[
\left\| F\right\| _{k,p}=\left( \bE\left( |F|^{p}\right)
+\sum_{j=1}^k \mathbb{E}\left( \left\| D^j F\right\|
_{\mathcal{H}^{\otimes j}}^{p}\right) \right) ^{\frac{1}{p}},
\]
and
\[
\mathbb{D}^{\infty}(\mathcal{H})=\bigcap_{p \geq 1} \bigcap_{k
\geq 1} \mathbb{D}^{k,p}(\mathcal{H}).
\]
We also introduce the localized spaces $\bD^{k,p}_{\rm loc}
(\mathcal{H})$ by saying that a random variable $F$ belongs to
$\bD^{k,p}_{\rm loc} (\mathcal{H})$ if there exists a sequence of sets
$\Omega_n \subset \Omega$ and random variables $F_n \in \bD^{k,p}
(\cH)$ such that $\Omega_n \uparrow \Omega$
almost surely and such that $F = F_n$ on $\Omega_n$.

We then have the following key result which stems from Theorem
2.1.2 and Corollary 2.1.2. in \cite{Nualart2006}:

\begin{theorem}\label{theo:dens}
Let $F=(F_1,\ldots,F_n)$ be a $\cF$-measurable random
vector such that:
\begin{enumerate}
\item For every $i=1,\ldots,n$, $F_i \in \bD^{1,2}_{\rm loc}(\cH)$;
\item The Malliavin matrix
of the random vector $F$:
$
\Gamma= \left(  \langle D F^i , {D} F^j
\rangle_{\cH}  \right)_{1 \leq i,j \leq n}
$
is invertible almost surely.
\end{enumerate}
Then the law of $F$ has a density with respect to the Lebesgue
measure on $\bR^n$. 
\\
If moreover $F \in \bD^\infty (\cH)$ and, for
every $p >1$,
\[
\mathbb{E} \left( | \det \Gamma | ^{-p} \right) < +
\infty,
\]
then this density is smooth.
\end{theorem}

The following result is useful in the proof of  regularity for the solution of the stochastic differential equation \eqref{eq:nl}.
We 
recall here for the sake of completeness.
A proof can be found, for instance, in \cite[Lemma 1.5.3]{Nualart2006}.

\begin{proposition}
\label{p3}
Let $\{F_n\}$ be a sequence of variables in $\bD^{k,p}$ for some $p > 1$. 
Assume that the sequence $F_n$ converges to $F$ in
$L^p(\Omega)$ and that
\begin{equation}
\label{e4}
\sup_n \|F_n\|_{k,p} < \infty.
\end{equation}
Then $F$ belongs to $\bD^{k,p}$.
\end{proposition}




\section{The stochastic convolution process}
\label{scp}

We consider first the stochastic convolution term
\begin{align*}
Z(t,x) = \int_0^t \int_S \int_{\partial D} p_N(t-s,x,\bar y) \alpha(\sigma, \bar y) \,  {\rm d}\bar y \,{\rm d}B(\sigma,s)
\end{align*}
from equation \eqref{eq:nl2}.
For the sake of completeness we recall some basic facts about Wiener integral with respect to fractional Brownian motion. For more details see, for instance \cite{Bonaccorsi2008d}. In Section \ref{sec2} we have introduced the fractional Brownian motion as a Gaussian process with covariance function given by \eqref{cov}. It is useful having in mind another characterization of the process which will allows us to do all the required computations. 
\\
A fractional Brownian motion $B(\sigma, s)$ of Hurst parameter $H$ can be defined as the convolution product
\begin{equation*}
B(\sigma, t)=\int_0^t\int_SK_H(t-s){\rm d}W(\sigma,s)
\end{equation*}
where $\left\{W(\sigma, t), \sigma \in S, t \in \left[0,T\right]\right\}$ is a standard Brownian motion and $K_H$ is the kernel 
\begin{equation*}
K_H(t,s)=C_H(t-s)^{H-\frac 12}+C_H\left(\frac 12 -H\right)\int_s^t(u-s)^{H- \frac 32}\left(1- \left(\frac su\right)^{\frac 12 -H}\right) {\rm d}u.
\end{equation*}
where $C_H$ is a normalizing constant, given by 
\begin{equation}
C_H= \left(\frac{2H\Gamma\left(\frac 32-H\right)}{\Gamma\left(H+\frac 12\right)\Gamma(2-2H)}\right)^{\frac 12}
\end{equation}
Let us consider the integral
\begin{equation*}
\mathcal{I}(\varphi)=\int_0^t \int_S\varphi(\sigma, r){\rm d}B(\sigma, r)
\end{equation*}
then the It\^o isometry reads as
\begin{equation*}
\mathbb{E}|\mathcal{I}(\varphi)|^2= \int_0^t\int_S|K^*(\varphi(\sigma, \cdot))(s)|^2 \mu({\rm d}\sigma){\rm d}s
\end{equation*}
for an operator $K^*$ which maps the reproducing kernel Hilbert space $\mathcal{H}$ into $L^2(0,T)$:
\begin{equation*}
(K^*\varphi)(s)=\pmb{1}_{(0,t)}(s)\int_s^t\varphi(\cdot, r)\left(\frac sr\right)^{\frac 12 -H}(r-s)^{H- \frac 32}{\rm d}r.
\end{equation*}

\subsection{Global regularity and smoothness in space of the convolution process}

Next proposition is related to space regularity of the process $Z(t, x)$.

\begin{proposition}
\label{cont_convoluzione}
For any $t \in [0,T]$, there exists a version of the stochastic convolution process 
\begin{align*}
Z(t,x) = \int_0^t \int_S \int_{\partial D} p_N(t-s,x,\bar y) \alpha(\sigma,\bar y) \, {\rm d}\bar y \, {\rm d}B(\sigma,s)
\end{align*}
that is a continuous function on $D$ and locally H\"older continuous, of arbitrary exponent $\gamma < 1$.
\end{proposition}

\begin{proof} The thesis follows from Kolmogorov's criterium of continuity and the Gaussian character of the stochastic convolution process.
Let for simplicity be $\phi(t,x,\sigma) = \int_{\partial D} p_N(t,x,\bar y) \alpha(\sigma,\bar y) \, {\rm d}\bar y$;
first, we have that
\begin{align*}
\bE \left| Z(t,x) \right|^2 = \int_0^t \int_S \left| \left(K^*(\phi(t-\cdot,x,\sigma) \uno_{(0,t)}(\cdot)) \right)(s) \right|^2 \, \mu({\rm d}\sigma) \, {\rm d}s
\end{align*}
Fix $x \in D$ and $M$ a convex neighbourhood of $x$ strictly contained in $D$,
$\mathop{\rm dist}(M,\partial D) > \varepsilon$. For any $z \in M$ it follows
\begin{align*}
\bE &\left| Z(t,x) - Z(t,z)\right|^2 = \int_0^t \int_S \left| \left(K^*([\phi(t-\cdot,x,\sigma)- \phi(t-\cdot,z,\sigma)] \uno_{(0,t)}(\cdot)) \right)(s) \right|^2 \, \mu({\rm d}\sigma) \, {\rm d}s
\\
&=
\int_0^t \int_S \left| \int_s^t \left(\frac{s}{r}\right)^{1/2-{H}} (r-s)^{{H}-3/2} [\phi(t-r,x,\sigma)- \phi(t-r,z,\sigma)] \, {\rm d}r \right|^2  \, \mu({\rm d}\sigma) \, {\rm d}s
\\
&\le t^{2H-1} \int_0^t \int_S s^{1-2{H}} \left| \int_s^t (r-s)^{{H}-3/2} [\phi(t-r,x,\sigma)- \phi(t-r,z,\sigma)] \, {\rm d}r \right|^2  \, \mu({\rm d}\sigma) \, {\rm d}s
\end{align*}
We may estimate, by using Corollary \ref{c.app1} for some $\mu \in \left(\frac12,H\right)$  (that is always possible, since $H > \frac12$ by assumption)
\begin{align*}
|\phi(t-r,x,\sigma)-& \phi(t-r,z,\sigma)| = \left| \int_{\partial D} [p_N(t-r,x,\bar y) - p_N(t-r,z,\bar y)] \alpha(\sigma, \bar y) \, {\rm d}\bar y \right|
\\
&\le \int_{\partial D} \int_0^1 \left| \frac{\partial}{\partial u}p_N(t-r, x + u(z-x), \bar y) \right| \, {\rm d}u \, \alpha(\sigma, \bar y) \, {\rm d}\bar y
\\
&\le |z-x| \int_{\partial D} \int_0^1 \left| \nabla p_N(t-r, x + u(z-x), \bar y) \right| \, {\rm d}u \, \alpha(\sigma, \bar y) \, {\rm d}\bar y
\\
&\le 
(t-r)^{-\mu} \, |z-x| \, \int_{\partial D} \left( \int_0^1 |x + u(z-x) - \bar y|^{-(d+1)+2\mu} \, {\rm d}u \right) \alpha(\sigma, \bar y) \, {\rm d}\bar y
\\
&\le 
(t-r)^{-\mu} \, |z-x| \, \varepsilon^{-(d+1)+2\mu} \, \int_{\partial D} \alpha(\sigma, \bar y) \, {\rm d}\bar y
\end{align*}
where we recall that $\varepsilon > 0$ is the distance between $M$ and $\partial D$.
%
Thus we obtain
\begin{align*}
\bE &\left| Z(t,x) - Z(t,z)\right|^2 
\\
&\le
t^{2H-1} \int_0^t \int_S s^{1-2{H}} \left| \int_s^t (r-s)^{{H}-3/2} \left[ (t-r)^{-\mu} \, |z-x| \, \varepsilon^{-(d+1)+2\mu} \, \int_{\partial D} \alpha(\sigma, \bar y) \, {\rm d}\bar y \right] \, {\rm d}r \right|^2  \, \mu({\rm d}\sigma) \, {\rm d}s
\\
&\le
|z-x|^2 \, \varepsilon^{-2(d+1)+4\mu} \, \|\alpha\|^2_{L^2(\partial D \times S)}
t^{2H-1} \int_0^t  s^{1-2{H}} \left| \int_s^t (r-s)^{{H}-3/2} (t-r)^{-\mu}  \, {\rm d}r \right|^2  \, {\rm d}s
%
%
\end{align*}
%
%
%
%
%
an explicit computation for the inner integral leads to
\begin{align*}
\bE \left| Z(t,x) - Z(t,z)\right|^2 
\le&
C_{H,\mu}^2 \, |z-x|^2 \, \varepsilon^{-2(d+1)+4\mu} \, \|\alpha\|^2_{L^2(\partial D \times S)}
t^{2H-1} \int_0^t  s^{1-2{H}} \left(t-s \right)^{2H-2\mu-1}  \, {\rm d}s
\end{align*}
%
%
%
%
%
hence
\begin{align*}
\bE \left| Z(t,x) - Z(t,z)\right|^2 
\le&
C_{H,\mu,\alpha,\varepsilon} \, t^{2H-2\mu} \, |z-x|^2.
\end{align*}
Taking into account that $Z(t,x) - Z(t,z)$ is a centered Gaussian random variable with variance bounded by the right-hand side of previous
inequality,
we have for any $k \in \bN$
\begin{align*}
\bE \left| Z(t,x) - Z(t,z)\right|^{2k}
\le&
C_{k,T,H,\mu,\alpha,\varepsilon}  \, |z-x|^{2k}.
\end{align*}
For $k$ large enough, writing $2k = d + (2k-d)$, we appeal to Kolmogorov's continuity theorem for random fields to conclude that
there exists a modification of $Z(t,\cdot)$ that is continuous on $M$ and $\gamma$-H\"older continuous for arbitrary $\gamma < \frac{2k-d}{2k}$ hence,
sending $k$ to infinity, for arbitrary $\gamma < 1$.
\end{proof}

\begin{remark}
\label {regolar_Z}
Under condition \textbf{(a1)} it can be easily shown that, for every $(t, x) \in \left[0,T\right] \times D$, $Z(t,x)$ is an $L^p(\Omega)$ 
random variable, for every $p \ge 2$. 
The proof follows by using almost the same computations used in the proof of Proposition \ref{cont_convoluzione}. 
\end{remark}

In order to give a meaning to the nonlinear boundary conditions, we shall prove that this process has a trace
on the boundary $\partial D$, compare \cite[Lemma 2.7]{Barbu2014}.

\begin{definition}
If $E$ is a separable Banach space, we denote by $L^q_W(0,T;E)$ the space $L^q(0,T;L^q(\Omega, \mathcal{F},\mathbb{P};E))$ of all $\left\{\mathcal{F}_t\right\}_{t \ge 0}$-adapted mappings  $X:(0,T) \rightarrow L^q(\Omega, \mathcal{F}, \mathbb{P};E)$ with the norm 
\begin{equation*}
\|X\|_{L^q_W(0,T;E)}=\left(\int_0^T\mathbb{E}|X(t)|^q_E \ {\rm d}t\right)^{\frac 1q}
\end{equation*}
where $\mathbb{E}$ is the expectation in the probability space $(\Omega, \mathcal{F}, \mathbb{P})$.
\end{definition}

\begin{lemma}
\label{l6.X}
Let us assume the non degeneracy hypothesis \textbf{(a1')},
then the boundary trace of the stochastic convolution process $Z(t,\xi)$ belongs to $L^p_W(0,T;L^p(\partial D))$.
\end{lemma}

\begin{proof}
In order to prove the theorem we have to show that the following quantity is finite:
\begin{align*}
\|\tau(Z)\|^p_{L^p_W(0,T;L^p(\partial D))}
=\int_0^T\mathbb{E}\|\tau(Z)\|^p_{L^p(\partial D)}\, {\rm d}t
= \int_0^T\int_{\partial D}\mathbb{E}|Z(t,\xi)|^p\,{\rm d} \xi \, {\rm d}t.
\end{align*}
So it is sufficient to prove that 
\begin{align*}
\mathbb{E}|Z(t,\xi)|^p\le C < \infty,
\end{align*}
where the constant $C$ does not depend on the variables $t$ and $\xi$.
Using the same notation as in the proof of Proposition \ref{cont_convoluzione}, let us write 
\begin{align*}
Z(t, \xi)=\int_0^t \int_S \phi(t-s, \xi, \sigma) \, {\rm d}B(\sigma, s) = \mathcal{I}(\phi)(t, \xi).
\end{align*}
Recalling now that the Burkholder-Davis-Gundy's inequality reads, for the fractional Brownian motion case, as
\begin{align*}
\mathbb{E}|\mathcal{I}(\phi)(t, \xi)|^p 
\le \left(\int_0^t\int_S
|\left(K^*(\phi(t- \cdot, \xi,\sigma)\pmb{1}_{(0,t)}(\cdot))\right) (s)|^2 \, \mu({\rm d}\sigma) \, {\rm d}s \right)^{\frac p2},
\end{align*}
and using the estimate \eqref{stima_ok}, we get
\begin{align*}
\mathbb{E}|Z(t, \xi)|^p &= \mathbb{E}|\mathcal{I}(\phi)(t, \xi)|^p= \mathbb{E}\left|\int_0^t \int_S\phi(t-s, \xi,\sigma){\rm d}B(\sigma, s)\right|^p \\
&\le  \left(\int_0^t\int_S|\left(K^*(\phi(t- \cdot, \xi,\sigma)\pmb{1}_{(0,t)}(\cdot))\right) (s)|^2\mu({\rm d}\sigma){\rm d}s\right)^{\frac p2}\\
&\le \left(\int_0^t \int_S\left|\int_s^t\left(\frac {s}{r}\right)^{\frac12 -H}(r-s)^{H- \frac 32}\phi(t-r, \xi, \sigma){\rm d}r \right|^2 \mu ({\rm d}\sigma) {\rm d} s \right)^{\frac p 2}\\
&=\left(\int_0^t \int_S\left|\int_s^t\left(\frac {s}{r}\right)^{\frac12 -H}(r-s)^{H- \frac 32}\left( \int_{\partial D}p_N(t-r, \xi, \bar{y})\alpha(\sigma, \bar{y})\, {\rm d}\, \bar{y}\right){\rm d}r \right|^2 \mu ({\rm d}\sigma) {\rm d} s \right)^{\frac p 2}\\
&\le \left(\int_0^t \int_S\left|\int_s^t\left(\frac {s}{r}\right)^{\frac12 -H}(r-s)^{H- \frac 32}\left( \int_{\partial D}c(t-r)^{- \mu} |\xi-\bar y|^{2 \mu-d}\alpha(\sigma, \bar{y}){\rm d} \bar{y}\right){\rm d}r \right|^2 \mu ({\rm d}\sigma) {\rm d} s \right)^{\frac p 2}\\
&=c\left(\int_0^t\left(\int_S\left(\int_{\partial D}|\xi-\bar y|^{2 \mu-d}\alpha(\sigma, \bar{y}) {\rm d}\bar{y}\right)^2\mu({\rm d}\sigma)\right)\left|\int_s^t\left(\frac{s}{r}\right)^{\frac 12 -H}(r-s)^{H- \frac 32}(t-r)^{- \mu}{\rm d}r\right|^2{\rm d}s \right)^{\frac p2}\\
&=c\left(\int_S\left(\int_{\partial D}|\xi-\bar y|^{2 \mu-d}\alpha(\sigma, \bar{y}) {\rm d}\bar{y}\right)^2\mu({\rm d}\sigma)\right)^{\frac p2}\left(\int_0^t\left|\int_s^t\left(\frac{s}{r}\right)^{\frac 12 -H}(r-s)^{H- \frac 32}(t-r)^{- \mu}{\rm d}r\right|^2{\rm d}s \right)^{\frac p2}\\
\end{align*}
With the same computations used in the proof of Proposition \ref{cont_convoluzione}, we recover the following estimate for the time variable integral
$$ \left(\int_0^t\left|\int_s^t\left(\frac{s}{r}\right)^{\frac 12 -H}(r-s)^{H- \frac 32}(t-r)^{- \mu}{\rm d}r\right|^2{\rm d}s \right)^{\frac p2}\le C_{\mu, H, p}T^{p(H-\mu)}$$
for $\mu \in (\frac 12, H)$.

As regards the space variable integral, using Lemma \ref{prel_lemma}, we get
$$\left(\int_S\left(\int_{\partial D}|\xi-\bar y|^{2 \mu-d}\alpha(\sigma, \bar{y}) {\rm d}\bar{y}\right)^2\mu({\rm d}\sigma)\right)^{\frac p2}
\le C \|\alpha\|^p_{L^{\theta}(\partial D;L^2(S))}.
$$
Hence we have 
$$\mathbb{E}|Z(t, \xi)|^p \le C_{\mu, p, H}T^{p(H- \mu)}\|\alpha\|^p_{L^{\theta}(\partial D;L^2(S))}.$$
Since the constant does not depend on $t$ and $\xi$ we get the thesis.
\end{proof}

\begin{remark}
\label{rem3}
Let us notice that we have the continuity in space of the convolution process only inside the domain $D$. For what concerns the spatial regularity of the process on the boundary we can only give a meaning of its trace as an element of $L^p(\partial D)$. So, on the boundary, we do not have enough regularity in order to evaluate the process pointwise in space. Anyway, as pointed out in the Introduction, in the following we will consider $Z(t,\xi)$ for $\xi \in \partial D$ (and then $u(t, \xi)$), 
remembering that these are only representative elements in the class $L^p(\partial D))$, and then all the stated results are valid only a.e. in space.
\end{remark}

\subsection{Malliavin derivative of the convolution process and existence of a density}

We next analyze the Malliavin derivative of the process $Z(t,x)$
\begin{align*}
D_{s,\sigma} Z(t,x)  = \int_{\partial D} p_N(t-r,x,y) \alpha(\sigma, \bar y){\mathds 1}_{\left[0,t\right]}(r) {\rm d}\bar y.
\end{align*}

As explained in Remark \ref{rem3}, we can study pointwise in space the Malliavin derivative of the random variable $Z(t, x)$ for every $x$ inside the domain. 
But, as a necessary step, we need to study the Malliavin derivative computed for points on the boundary, and this study has to be interpreted in an a.e. sense.

\begin{lemma}
\label{lem_5}
Under assumption \textbf{(a1')}, the random variable $Z(t,x)$ belongs to $\bD^{1,2}$ for any $(t,x) \in (0,T) \times \bar{D}$.
\end{lemma}

\begin{proof}
In view of  Lemma \ref{l6.X}, Proposition \ref{cont_convoluzione}, and Remark \ref{regolar_Z}, 
it remains to prove that the Malliavin derivative belongs to $\cH$.
\\
The key point in the following computation is provided by the estimate
\eqref{stima_ok}
\begin{align*}
p_N(t,x,\bar y) \le C \, t^{-\mu} \, |x-\bar y|^{-d+2\mu} 
\end{align*}
\\
We get
\begin{align*}
\|D \, Z(t,x)\|^2_{\cH} =& \int_0^t \int_0^t \int_S D_{s,\sigma}Z(t,x) \, D_{r,\sigma}Z(t,x) \, |s-r|^{2H-2} \, \mu({\rm d}\sigma) \, {\rm d}r \, {\rm d}s
\\
=& \int_0^t \int_0^t \int_S \left( \int_{\partial D} p_N(t-s,x,\bar y) \alpha(\sigma,\bar y) \, {\rm d}\bar y \right) \, \left( \int_{\partial D} p_N(t-r,x,\bar y) \alpha(\sigma,\bar y) \, {\rm d}\bar y \right) \, |s-r|^{2H-2} \, \mu({\rm d}\sigma) \, {\rm d}r \, {\rm d}s
\\
\le& C \, \int_0^t \int_0^t \int_S \left( \int_{\partial D} |x-\bar y|^{-d+2\mu} |\alpha(\sigma,\bar y)| \, {\rm d}\bar y \right)^2 \, (t-s)^{-\mu} (t-r)^{-\mu} \, |s-r|^{2H-2} \, \mu({\rm d}\sigma) \, {\rm d}r \, {\rm d}s
\end{align*}
hence we can separately examine the spatial and the time integrals. First, we have
\begin{align*}
\int_0^t \int_0^t (t-s)^{-\mu} (t-r)^{-\mu} |s-r|^{2H-2} \, {\rm d}r \, {\rm d}s
=& 2 \int_0^t s^{-\mu} \, \left( \int_0^s r^{-\mu}  (s-r)^{2H-2} \, {\rm d}r \right) \, {\rm d}s
\\
=& 2 \int_0^t s^{-\mu}  \left( C_{H,\mu} s^{2H-1-\mu } \right) \, {\rm d}s
= C'_{H,\mu} t^{2H-2\mu}
\end{align*}
In order to handle the spatial term, we have to consider separately the cases $x \in \mathring{D}$ and $x \in \partial D$. 
Let us at first consider the case $x \in \mathring{D}$. 
In this case $\bar y \mapsto |x-\bar y|^{-d+2\mu}$ belongs to $L^2(\partial D)$. 
In fact, denoting 
$\varepsilon = \text{dist}(x,\partial D)$ and $\Gamma = \text{diam}(D)$ the diameter of $D$, then
\begin{align*}
\varepsilon \le |x - \bar y| \le \Gamma
\end{align*}
and so
 \begin{align}
\label{spatial_term}
\int_{\partial D} |x-\bar y|^{ 2(-d+2\mu)} \, {\rm d}\bar y &\le \int_{\partial D} \max\{\varepsilon^{-2d+4\mu},\Gamma^{-2d+4\mu}\} \, {\rm d}\bar y
< +\infty.
\end{align}
Therefore, we get
\begin{align*}
\int_S \left( \int_{\partial D} |x-\bar y|^{-d+2\mu} |\alpha(\sigma,\bar y)| \, {\rm d}\bar y \right)^2 \, \mu({\rm d}\sigma)
\le C \, \|\alpha\|^2_{L^2( \partial D \times S)}
\end{align*}

If $x \in \partial D$, thanks to Lemma \ref{prel_lemma} we have that 
\begin{align*}
\int_S \left( \int_{\partial D} |x-\bar y|^{-d+2\mu} |\alpha(\sigma,\bar y)| \, {\rm d}\bar y \right)^2 \, \mu({\rm d}\sigma)
\le C \, \|\alpha\|^2_{L^{\theta}( \partial D;L^2(S)}
\end{align*}
which allows to conclude
\begin{align*}
\|D \, Z(t,x)\|^2_{\cH} < +\infty
\end{align*}
as required. 
\end{proof}

The following result is a refinement of the previous Lemma.

\begin{lemma}
\label{lemma14}
The random variable $Z(t,x)$ belongs to $\mathbb{D}^{1,p}$ for any $(t,x) \in (0,T)\times \bar{D}$ and for every $p>2$.
\end{lemma}

\begin{proof}
We have only to prove that 
\begin{align*}
 \mathbb{E}\|DZ(t,x)\|^p_{\mathcal{H}} < \infty.
\end{align*}
This easily follows from Lemma \ref{lem_5}. In fact
\begin{align*}
\mathbb{E}\|DZ(t,x)\|^p_{\mathcal{H}}=\|DZ(t,x)\|^p_{\mathcal{H}}=\left(\|DZ(t,x)\|^2_{\mathcal{H}}\right)^{\frac p2}
\end{align*}
and so, if $x \in D$ we recover
\begin{align*}
\mathbb{E}\|DZ(t,x)\|^p_{\mathcal{H}}\le \left(C'_{H,\mu}T^{2H-2 \mu}C\|\alpha\|^2_{L^2(\partial D \times S)}\right)^{\frac p2}< \infty,
\end{align*}
whereas for the case $\xi \in \partial D$ we get
\begin{align*}
\mathbb{E}\|DZ(t,\xi)\|^p_{\mathcal{H}}\le \left(C'_{H,\mu}T^{2H-2 \mu}C_{|\partial D|}\right)^{\frac p2}< \infty.
\end{align*}
\end{proof}

%
%

\begin{lemma}
The stochastic convolution process belongs to $\mathbb{D}^{\infty}$.
\end{lemma}
\begin{proof}
Simply notice that $D\delta(u)=u$ for any deterministic function $u \in \mathcal{H}$, where we use the notation $\delta(u)$ for the Wiener integral. Therefore, higher order derivatives vanishes and the thesis follows.
\end{proof}

We finally prove the existence of the density of the random variable $Z(t,x)$ with respect to the Lebesgue measure on $\mathbb{R}$. We shall use the criterion for absolute continuity stated in Theorem \ref{theo:dens}.

\begin{lemma}
\label{esis_dens_Z}
Under the non-degeneracy hypothesis \textbf{(a2)}, the random variable $Z(t,x)$, $(t, x) \in \left[0,T\right] \times D$ has a smooth density with respect to the Lebesgue measure on $\mathbb{R}$.
\end{lemma}

\begin{proof}
The thesis follows once we provide a constant $\delta$ that satisfies the estimate
$$\|DZ(t,x)\|^2_{\mathcal{H}} \ge \delta >0 \qquad \text{a.s.}$$
So we compute
\begin{align*}
\|D \, Z(t,x)\|^2_{\cH} =& \int_0^t \int_0^t \int_S D_{s,\sigma}Z(t,x) \, D_{r,\sigma}Z(t,x) \, |s-r|^{2H-2} \, \mu({\rm d}\sigma) \, {\rm d}r \, {\rm d}s
\\
=& \int_0^t \int_0^t \int_S \left( \int_{\partial D} p_N(t-s,x,\bar y) \alpha(\sigma,\bar y) \, {\rm d}\bar y \right) \, \left( \int_{\partial D} p_N(t-r,x,\bar y) \alpha(\sigma,\bar y) \, {\rm d}\bar y \right) \, |s-r|^{2H-2} \, \mu({\rm d}\sigma) \, {\rm d}r \, {\rm d}s
\\
&\ge C_1^2 \int_0^t \int_0^t (t-s)^{-\frac d2} (t-r)^{- \frac d2} |s-r|^{2H-2} \times
\\
&\qquad \times \quad \left[\int_S\left(\int_{\partial D} e^{-c_2 \frac{|x-\bar{y}|^2}{t-s}} \alpha(\sigma, \bar{y}){\rm d}\bar{y} \right)\left(\int_{\partial D}e^{-c_2 \frac{|x-\bar{y}|^2}{t-r}} \alpha(\sigma, \bar{y}){\rm d}\bar{y} \right)\mu({\rm d}\sigma)\right] 
{\rm d}r \, {\rm d}s
\end{align*}
where we have used the lower bound estimate for the Poisson kernel (see Lemma \ref{lower}). Let us now fix $\varepsilon>0$ and let $s< t- \varepsilon^2$. 
We then have the following estimate
\begin{align*}
\int_{\partial D}e^{-c_2\frac{|x- \bar{y}|^2}{t-s}}\alpha(s, \bar{y}){\rm d}\bar{y} &\ge \int_{\partial D \cap B(x, \varepsilon)}e^{-c_2\frac{|x- \bar{y}|^2}{t-s}}\alpha(s, \bar{y}) {\rm d}\bar{y}\\
&\ge c\alpha_0 \varepsilon^{d-1} e^{-\frac{\varepsilon^2}{t-s}} \ge c\alpha_0 \varepsilon^{d-1} e^{-c_2} =C_{\alpha_0,c_2} \varepsilon^{d-1}
\end{align*}
Choosing $\varepsilon^2=\frac t2$ we then have
\begin{align*}
\|D \, Z(t,x)\|^2_{\cH} \ge  
&C_{\alpha_0,c_2,c_1,|S|}t^{d-1} \int_0^{\frac t2} \int_0^{\frac t2} (t-s)^{-\frac d2} (t-r)^{- \frac d2} |s-r|^{2H-2} {\rm d}r {\rm d}s\\
&=C_{\alpha_0,c_2,c_1,|S|}t^{d-1} \int_0^{\frac t2} \int_0^{\frac t2} (\sigma + \frac t2)^{-\frac d2} (\rho+\frac t 2)^{- \frac d2} |\sigma- \rho|^{2H-2} {\rm d}\sigma {\rm d}\rho\\
&\ge C_{\alpha_0,c_2,c_1,|S|}t^{-1} \int_0^{\frac t2} \int_0^{\frac t2}  |\sigma- \rho|^{2H-2} {\rm d}\sigma {\rm d}\rho \ge \tilde{C} t^{2H-1} >0
\end{align*}
as required.
\end{proof}

\section{The solution of the nonlinear problem}
\label{Eos}
We consider in this section the non homogeneous diffusion equation \eqref{e1}. 
In the first part of the section we impose condition {\bf (g1)} on the function $g$:
\begin{multline}\label{eq:ass1}
\text{$g$ is Lipschitz continuous and belongs to the class $C^1(\mathbb{R})$ with $|\partial_ug(u)|\le L,$}\notag \\
\text{$g(u) \le c(1+|u|)$ for some finite constant $L > 0$. }
\end{multline}
%

\subsection{Existence and global regularity of the solution}

While dealing with the problem of the existence and uniqueness of the solution, we shall apply the same ideas of \cite[Section 3.2]{Barbu2014}.
We first consider the Volterra equation on $\left[0,T\right] \times \partial D$,
\begin{equation}
\label{eq_bordo}
u(t, \xi)=\int_0^t \int_{\partial D} p_N(t-s,\xi, \bar{y})g(u(s, \bar{y})) \, {\rm d}\bar{y} \, {\rm d}s
+ \int_0^t \int_S \int_{\partial D} p_N(t-s,\xi,\bar{y}) \alpha(\sigma, \bar{y}) \, {\rm d} \bar y \, {\rm d}B(\sigma, s).
\end{equation}

\begin{theorem}
Assume that Hypothesis \textbf{(g1)} holds. Then equation \eqref{eq_bordo} has a unique solution in the space $L^p_W(\left[0,T\right];L^p(\partial D)).$
\end{theorem}

\begin{proof}
For the proof we follow the ideas of \cite[Lemma 3.4]{Barbu2014}.
The existence of a solution follows from a fixed point argument. 
Let us introduce in the space $L^p_W(\left[0,T\right];L^p(\partial D))$ the equivalent norm
\begin{equation*}
\|f\|_{L^p_W}= \mathbb{E}\int_0^T e^{-\lambda t}\|f(t)\|^p_{L^p(\partial D)}\, {\rm d}t
\end{equation*} 
then we introduce the operator 
\begin{equation*}
\Phi(u)(t, \xi)=\int_0^t \int_{\partial D} p_N(t-s,\xi, \bar{y})g(u(s, \bar{y})) \, {\rm d}\bar{y} \, {\rm d}s + Z(t, \xi)
\end{equation*}
and prove that it is a well-defined mapping from the space $L^p_W(\left[0,T\right];L^p(\partial D))$ into itself. Moreover, it is a contraction for some suitable $\lambda$. 
Thanks to Lemma \ref{l6.X}, this will follow from the following estimate

\begin{align*}
\|\Phi(u)-\Phi(v)\|^p_{L^p_W}&=\mathbb{E}\int_0^T e^{-\lambda t} \int_{\partial D}\left|\int_0^t\int_{\partial D} p_N(t-s, \xi, \bar y)\left(g(u(s, \bar y))-g(v(s,\bar y)) \right)\, {\rm d} \bar y \, {\rm d} s\right|^p \, {\rm d} \xi \, {\rm d}t
\\
&\le L^p \mathbb{E}\int_0^T e^{-\lambda t} \int_{\partial D}\left(\int_0^t\int_{\partial D} p_N(t-s, \xi, \bar y)\left|u(s, \bar y)-v(s,\bar y) \right|\, {\rm d} \bar y \, {\rm d} s\right)^p \, {\rm d} \xi \, {\rm d}t
\end{align*}

Now, proceeding exactly as in the proof of Lemma \ref{l-dopo}, 
we recover 
\begin{align*}
\|\Phi(u)-\Phi(v)\|^p_{L^p_W} \le C\lambda^{\frac{\mu-1}{p-1}} \|u-v\|^p_{L^p_W}
\end{align*}

Then we see that there exists $\lambda$ large enough such that $C\lambda^{\frac{\mu-1}{p-1}}< 1 -\varepsilon <1$ and this proves the Theorem.
\end{proof}

Given the process $\varphi \in L^p_W(\left[0,T\right];L^p( \partial D))$ 
that is the solution of problem \eqref{eq_bordo}, the solution of the original problem \eqref{eq:nl} is given by the representation formula 
\begin{equation}
\label{sol_int}
u(t,x)=\int_0^t \int_{\partial D} p_N(t-s,x, \bar{y})g(\varphi(s, \bar{y})) \, {\rm d}\bar{y} \, {\rm d}s
+ \int_0^t \int_S \int_{\partial D} p_N(t-s,x,\bar{y}) \alpha(\sigma, \bar{y}) \, {\rm d}\bar{y} \, {\rm d}B(\sigma, s).
\end{equation}

\begin{corollary}
For every $t \in \left[0,T \right]$ and almost surely, the solution $u(t,x)$ of problem \eqref{eq:nl} is a continuous function for $x \in D$.
\end{corollary}

\begin{corollary}
There exists a continuous modification $u=\left\{u(t,x), t \in \left[0,T\right] \times D\right\}$ of the solution process \eqref{sol_int}.
\end{corollary}

\subsection{The Malliavin derivative of the solution}

We are concerned with the law of the random variable $u(t,x)$ that represents the solution of problem \eqref{eq:nl} for $t \in (0,T)$ and $x \in D$.
We shall prove first that $u(t,x)$ has a density that is absolutely continuous with respect to the Lebesgue measure.
Later in this section we shall prove, under additional conditions on the coefficients of \eqref{eq:nl}, that this density is smooth ($C^\infty$).

Next result concerns the existence of the Malliavin derivative for the solution $u$. Heuristically, the
Malliavin derivative is the solution of the problem that we get by formally taking derivative in the original problem \eqref{eq:nl2};
however, to make the construction complete, we shall appeal to the Picard's approximations of $u$ and apply Proposition \ref{p3}.

\begin{theorem}
\label{D12_sol}
Assume that Hypothesis {\bf (g1)} {{and \textbf(a1')}} hold. 
Then the Malliavin derivative of the solution process $Du$ belongs to the space $L^p_{W}(\left(0,T\right) \times \partial D; \mathcal{H})$ and it satisfies the following equation:
\begin{multline}
\label{sol_bordo}
D_{r,\sigma}u(t,\xi)= \int_{\partial D}p_N(t-r,\xi,\bar{y})\alpha(\sigma,\bar{y}){\pmb 1}_{\left[0,t\right]}(r) \, {\rm d}\bar{y} 
\\
+ \int_r^t \int_{\partial D} p_N(t-s,\xi,\bar{y})\partial_ug(u(s,\bar{y}))D_{r,\sigma}u(s,\bar{y}) \, {\rm d}\bar{y} \, {\rm d}s.
\end{multline}
\end{theorem}

\begin{proof}
Let us consider the sequence of Picard's approximations 
\begin{align*}
u_0(t,\xi) &= \int_0^t \int_S \int_{\partial D} p_N(t-s,\xi,\bar y) \alpha(\sigma, \bar y) \,  {\rm d}\bar y \,{\rm d}B(\sigma,s),
\\
u_{n+1}(t,\xi) &=  u_0(t,\xi) +\int_0^t \int_{\partial D} p_N(t-s,\xi,\bar y) g(u_n(s,\bar y)) \, {\rm d}\bar y \, {\rm d}s.
\end{align*}
{{From Lemma \ref{l6.X} we know that $u_0  \in L^p_W(0,T; L^p(\partial D))$. Using then the Hypothesis $\textbf{(g1)}$ and appealing to Lemma \ref{l-dopo} (in the case $k=0$), immediately follows that 
\begin{align*}
\|u_{n+1}-u_n\| \le C \|u_n-u_{n-1}\|
\end{align*}
for a constant $C<1$. So we obtain that the sequence $u_n$ converge to $u$, which is clearly the solution of \eqref{eq:nl2}, in $L^p_W(0,T; L^p(\partial D))$ and $\sup_n\|u_n\| < \infty$.}} 

It remains to prove that
$$\sup_n \|D u_n(t,\xi)\| < \infty,$$ which allows to apply Proposition \ref{p3} and get the thesis.

For any $n \in \bN$, taking the Malliavin derivative in the equation defining $u_{n+1}$ we get 
\begin{multline*}
D_{r,\sigma} u_{n+1}(t,\xi) =  \int_0^t \int_{\partial D} p_N(t-s,\xi,\bar y) \partial_u g(u_n(s,\bar y)) \, D_{r,\sigma} u_n(s,\bar y) \, {\rm d}\bar y \, {\rm d}s
\\
+  \int_{\partial D}p_N(t-r,\xi,\bar{y})\alpha(\sigma,\bar{y}){\pmb 1}_{\left[0,t\right]}(r) \, {\rm d}\bar{y}.
\end{multline*}
%
Let us introduce on the space 
$L^p_{W}(\left(0,T\right) \times \partial D;\mathcal{H})$
the equivalent norm
\begin{equation}
\label{e:0207-1}
\|f\|^p_{L^p_{W}} = \mathbb{E} \int_0^Te^{-\lambda t}  \|f(t)\|^p_{L^p(\partial D; \mathcal{H})} {\rm d}t
\end{equation}
for some constant $\lambda$ to be chosen later.  This norm is equivalent to the standard one. 
Let us define 
\begin{equation}
\begin{aligned}
\label{PIC}
\varphi_{0}(r, \sigma, t, \xi) &:= \int_{\partial D}p_N(t-r,\xi,\bar y)\alpha(\sigma, \bar y){\pmb 1}_{\left[0,t\right]}(r) \, {\rm d}\bar y
\\
\varphi_{n+1}(r, \sigma, t, \xi) &:= \varphi_{0}(r, \sigma, t, \xi) + \int_r^t\int_{\partial D}p_N(t-s,\xi,\bar y)\partial_ug(u_n(s,\bar y))\varphi_n(r, \sigma,s,\bar y) \, {\rm d}\bar y \,{\rm d}s.
\end{aligned}
\end{equation}
We aim to prove that, for every $n$, $\varphi_n \in L^p_{W}(\left[0,T\right] \times \partial D; \mathcal{H})$ and 
\begin{align}
\label{contraction}
\|\varphi_{n+1}-\varphi_{n}\|_{L^p_{W}(\left[0,T\right] \times \partial D; \mathcal{H})} \le C \|\varphi_{n}-\varphi_{n-1}\|_{L^p_{W}(\left[0,T\right] \times \partial D; \mathcal{H})}
\end{align}
for a suitable constant $C < 1$. Since this result is mainly a technical tool and requires some computation, we pospone its proof to Lemma \ref{l-dopo} in the Appendix.
Now, assume that \eqref{contraction} is given. Then 
\begin{align*}
\|\varphi_{n+1}\| 
&\le \|\varphi_{n+1} - \varphi_0\| + \|\varphi_0\| = \sum_{k=0}^n \|\varphi_{k+1} - \varphi_k\| + \|\varphi_0\| \\
&\le \sum_{k=0}^n C^{k} \|\varphi_0\| + \|\varphi_0\| 
\le \left(\frac{1}{1-C}+ 1\right)\|\varphi_0\| < \infty
\end{align*}
holds for any $n$. This allows to apply Proposition \ref{p3} and prove that $u \in \bD^{1,p}$. Moreover, since $D u_n$ converges weakly to $D u$, then
we pass to the limit in the expression of the Malliavin derivative of $u_n$ and we obtain equation \eqref{sol_bordo}. 
\end{proof}

\smallskip

The Malliavin derivative of the solution $u(t,x)$ for problem \eqref{eq:nl}, $x \in D$, is obtained from the process on the boundary, as explained in the 
following result.


\begin{proposition}
\label{def_D(t,x)}
Given the random variable 
$\varphi \in L^2_{W}(\left[0,T\right] \times \partial D; \mathcal{H})$ that is the solution of \eqref{sol_bordo}, the solution of 
\begin{equation*}
D_{r,z}u(t,x)= \int_{\partial D}p_N(t-r,x,\bar y)\alpha(\sigma, \bar y){\pmb 1}_{\left[0,t\right]}(r) \,{\rm d}\bar y + \int_r^t\int_{\partial D}p_N(t-s,x,\bar y)\partial_ug(u(s,\bar y))D_{r,z}u(s,\bar y) {\rm d}\bar y\,{\rm d}s
\end{equation*}
on $\left[0,T\right] \times D$ is given by the representation formula
\begin{equation}
\label{der_sol_int}
D_{r,z}u(t,x)= \int_{\partial D}p_N(t-r,x,\bar y)\alpha(\bar y){\pmb 1}_{\left[0,t\right]}(r) \,{\rm d}\bar y + \int_r^t\int_{\partial D}p_N(t-s,x,\bar y)\partial_ug(u(s,\bar y))\varphi(r, \sigma,s,\bar y){\rm d}\bar y\,{\rm d}s
\end{equation}
\end{proposition}

\subsection{Smoothness of the Malliavin derivative}

In the sequel we prove that, under the stronger assumption {\bf (g2)}, the solution of problem \eqref{eq:nl2} belongs to $\bD^\infty$.
In turn, this result will allow to study the smoothness of the density of the random variable $u(t,x)$ with respect to the Lebesgue measure.

Let us give an hint about the following construction. First, notice that $u \in \bD^\infty$ means that $u \in \bD^{k,p}$ for any $k \in \bN$ and $p \ge 1$;
moreover, it is also implied if we prove that $D^k u \in \bD^{1,p}$ for any $k$ and any $p$.
\\
Taking a look to the equation satisfied by the Malliavin derivative $D u(t,x)$, we notice that it has the same form as the original problem \eqref{eq:nl2},
just changing the relevant coefficients. Then we may use this analogy to prove that
this process has a derivative itself. After that, we shall use an iteration argument to conclude the construction.

For the sake of simplicity, we shall provide a unique result concerning existence of the Malliavin derivative of the solution for a general class of problems. 

\begin{lemma}
\label{l:0107-1}
Let $V(t,x)$ the solution of the linear (integral) equation in $\mathcal{H}^{\otimes k}$, for every $k \ge 1$,
\begin{equation}
\label{V}
V(t,x)=f_k(t,x)+\int_0^t \int_{\partial D} p_N(t-s,x, \bar{y})h(s,\bar{y})V(s,\bar{y}) {\rm d}s {\rm d}\bar{y}.
\end{equation}
Under the assumptions (for every $p \ge 2$)
\begin{equation}
\label{f_k}
f_k \in L^p_{W}((0,T) \times \partial D;\mathcal{H}^{\otimes k}),
\end{equation}
\begin{equation}
\label{h_ass}
h \in L^{\infty}_{W}((0, T) \times \partial D),
\end{equation}
\begin{equation}
\label{Df_k_ass}
Df_k\in L^p_W((0,T) \times \partial D; \mathcal{H}^{\otimes (k+1)}),
\end{equation}
\begin{equation}
\label{h_ass_2}
Dh \in L^{\infty}_W((0,T) \times \partial D; \mathcal{H}),
\end{equation}
we have that $V(t,x) \in \mathbb{D}^{1,\infty}$ for a.e. $(t,x)$
and 
\begin{equation}
\label{DV}
DV(t,x)=f_{k+1}(t,x)+ \int_0^T\int_{\partial D}p_N(t-s, x, \bar{y})h(s, \bar{y})DV(s, \bar{y}) {\rm d}s\, {\rm d}\bar{y},
\end{equation}
where 
\begin{equation}
f_{k+1}(t,x)=Df_k(t,x)+ \int_0^T \int_{\partial D} p_N(t-s, x, \bar{y})Dh(s, \bar{y}) V(s, \bar{y}){\rm d}\, \bar{y} {\rm d}s
\end{equation}
\end{lemma}

\begin{proof}
For the sake of simplicity let us divide the proof in four steps.

\textbf{Step 1:}
Consider the Picard approximations defined by the recursive equations:
\begin{align*}
V_0(t,x)=f_k(t,x);
\end{align*}

\begin{align}
\label{Picard_V}
V_{n+1}(t,x)=V_0(t,x)+\int_0^T\int_{\partial D}p_N(t-s, x, \bar{y})h(s, \bar{y})V_n(s, \bar{y}) {\rm d}s \,{\rm d}\bar{y}
\end{align}
Notice that the term $V_0 \in L^p_{W}((0,T) \times \partial D;\mathcal{H}^{\otimes k})$ by assumption.
Since $h$ is bounded and $p_N$ is a Gaussian kernel, it is possible to prove (use Lemma \ref{l-dopo}) that
the right hand side of equation \eqref{Picard_V} defines a contraction in the space $L^p_W((0,T)\times \partial D; \mathcal{H}^{\otimes k})$, for every $p \ge 2$. 
This implies the existence and uniqueness of the random variable $V \in \bigcap_{p \ge 2}L^p_W((0,T) \times \partial D; \mathcal{H}^{\otimes k})$
which is the solution of \eqref{V}. 

\textbf{Step 2:}
Taking the Malliavin derivative in \eqref{Picard_V} we get
\begin{align}
\label{e:0207-3}
DV_{n+1}(t,x)
=& Df_k(t,x)+ \int_0^t\int_{\partial D}p_N(t-s, x, \bar{y})Dh(s, \bar{y})V_n(s, \bar{y}) \, {\rm d}\bar{y} \, {\rm d}s
\notag \\
&+ \int_0^t\int_{\partial D}p_N(t-s, x, \bar{y})h(s, \bar{y})DV_n(s, \bar{y}) \, {\rm d}\bar{y} \, {\rm d}s
\notag \\
=& \lambda_n(t,x)+ \int_0^t\int_{\partial D}p_N(t-s, x, \bar{y})h(s, \bar{y})DV_n(s, \bar{y}) \, {\rm d}\bar{y} \, {\rm d}s.
\end{align}
We see that, for every $n$, $\lambda_n \in L^p_W((r,T) \times \partial D; \mathcal{H}^{\otimes (k+1)})$. This follows from assumptions \eqref{f_k} - \eqref{Df_k_ass} and by the fact (proved in Step 1) that $V_n \in L^p_W((r,T)\times \partial D; \mathcal{H}^{\otimes k})$ for every $n$.
So, proceeding exactly as in step 1 we get that the right hand side of the above expression defines a contraction in the space $L^p_W((0,T)\times \partial D; \mathcal{H}^{\otimes (k+1)})$.

\textbf{Step 3:}
At this point we recover the following estimate:
\begin{equation}
\label{sup_DV}
\sup_n\|DV_{n}\|_{L^p_W((0,T)\times \partial D; \mathcal{H}^{\otimes (k+1)})} \le C < \infty,
\end{equation}
In fact, we have:
\begin{equation}
\|DV_n\| \le\|DV_0\| + \|DV_n-DV_0\| \le \|DV_0\| + \sum_{j=0}^{n-1} \|DV_{j+1}-DV_j\|,
\end{equation}
where $\|DV_0\| < C$ since $DV_0 \in L^p_W((0,T)\times \partial D; \mathcal{H}^{\otimes (k+1)})$ by step 3 and, thanks to Lemma \ref{l-dopo} we recover
\begin{align*}
\|DV_{j+1}-DV_j\| 
&\le \|\lambda_j- \lambda_{j-1}\|+ \left|\left| \int_r^t\int_{\partial D}p_N(t-s, x, \bar{y})h(s, \bar{y})(DV_j-DV_{j-1})(s, \bar{y}){\rm d}\bar{y} {\rm d}s\right|\right|
\\
&\le C\|V_j-V_{j-1}\|+C\|DV_j - DV_{j-1}\|
\end{align*}
Iterating this inequality we recover
\begin{align*}
\sum_{j=0}^{n-1} \|DV_{j+1}-DV_j\| 
&\le \|V_1-V_0\|\sum_{j=0}^{n-1}jC^j + \|DV_1-DV_0\|\sum_{j=0}^{n-1} C^j
\\
&\le \frac{1}{(1-C)^2}\|V_1-V_0\|+ \frac{1}{1-C}\|DV_1-DV_0\|
\end{align*}
which allows us to conclude that estimate \eqref{sup_DV} holds.

\textbf{Step 4:} Appealing to Proposition \ref{p3}, by steps one and three it follows that $V \in \mathbb{D}^{1,p}$ for every $p \ge 2$. Since $V_n\rightarrow V$ and $DV_n \rightarrow DV$, passing to the limit in \eqref{e:0207-3} we get that $DV$ solves \eqref{DV}.

\end{proof}

\begin{remark}
\label{r:0107-1}
In our case, we recognize that
\begin{align*}
h(t,x) = \partial_u g(u(t,x))
\end{align*}
is a bounded process and, by considering \eqref{sol_bordo}, we may
set
\begin{align*}
f_1(t,x) = \int_{\partial D}p_N(t-r,\xi,\bar{y})\alpha(\sigma,\bar{y}){\pmb 1}_{\left[0,t\right]}(r) \, {\rm d}\bar{y} 
\end{align*}
which satisfies the assumption of Lemma \ref{l:0107-1}. We may also compute
\begin{align*}
f_2(t,x) = \int \int p_N(t-s,x,y) \partial^2_u g(u(s,y)) (D u(s,y))^2 \, {\rm d}y \, {\rm d}s
\end{align*}
and, in general, $f_{n+1}$ depends on $\partial^2_ug, \dots, \partial^{n+1}_ug$, and $D u, \dots, D^n u$
in a polynomial way, so a recursive argument implies that the assumption of Lemma \ref{l:0107-1} are satisfied for any $n$.
\end{remark}


\begin{corollary}
\label{cor}
The random variable $u(t, x)$ belongs to the space $\mathbb{D}^{\infty}$ for every $(t, x) \in \left[0,T\right] \times \bar{D}$.
\end{corollary}

\begin{proof}
First, as explained above, the result follows for $x \in \partial D$ from 
Lemma \ref{l:0107-1} and Remark \ref{r:0107-1}.
Then, we conclude by noticing that for $x \in D$ the Malliavin derivative is given (compare with Proposition \ref{def_D(t,x)}) in terms of the
same process on the boundary, hence we infer the regularity of $u(t,x)$ by that of $u(t,\xi)$ for $\xi \in \partial D$.
\end{proof}


\section{Existence of the density for the solution of the  non homogeneous equation}
\label{Edsnhe}

We are concerned with the law of the random variable $u(t,x)$ that represents the solution of problem \eqref{eq:nl} for $t \in (0,T)$ and $x \in D$.
We shall prove first that $u(t,x)$ has a density that is absolutely continuous with respect to the Lebesgue measure.
Later in this section we shall prove, under additional conditions on the coefficients of \eqref{eq:nl}, that this density is smooth ($C^\infty$).

\begin{theorem}
\label{exist_dens}
For every $t \in \left[0,T\right]$ and a.e. $\xi \in \partial D$, the random variable $u(t, \xi)$ has a density with respect to the Lebesgue measure on $\mathbb{R}$.
\end{theorem}

\begin{proof}
In order to get some relevant estimates , we will consider a smaller time interval than $(0,t)$ and consider the $\mathcal{H}$ norm of $u(t,\xi)$ on $(t-\delta,t)$ for some $\delta>0$ small enough.
Then we define, for every $\varphi \in \mathcal{H}$ the norm 
$$\|\varphi\|_{\mathcal{H}_{\delta}}:= \|\pmb{1}_{(t- \delta,t)}(\cdot)\varphi\|_{\mathcal{H}}.$$
It is then straightforward to get
$$\|\varphi\|_{\mathcal{H}} \ge \|\varphi\|_{\mathcal{H}_{\delta}}$$

The existence of the density follows from Theorem \ref{theo:dens} and the estimate
\begin{equation}
\label{21}
\mathbb{E}\|D u(t,\xi)\|^2_\cH > 0 \qquad \text{a.s.}
\end{equation}
In turn, to prove that \eqref{21} holds $\mathbb{P}$-a.s., the idea is to prove that 
\begin{align*}
\mathbb{P}\left(\|Du(t,\xi)\|^2_{\mathcal{H}} < \veps\right) \rightarrow 0
\end{align*}
as $\veps \to 0$.
\\
Setting for simplicity
\begin{align*}
G(t,\xi):=\int_0^t\int_{\partial D}p_N(t-s, \xi, \bar y)g(u(s, \bar y))\, {\rm d}\bar y\, {\rm d}s
\end{align*}
we get
\begin{align*}
\|Du(t,\xi)\|^2_{\mathcal{H}} \ge \|Du(t,\xi)\|^2_{\cH_\delta} &= \|DZ(t,\xi)+ DG(t,\xi)\|^2_{\mathcal{H}_\delta} \\
&\ge \frac12 \|DZ(t,\xi)\|^2_{\mathcal{H}_\delta}-\|DG(t,\xi)\|^2_{\mathcal{H}_\delta}.
\end{align*}
Using then Chebyshev's inequality we have
\begin{align*}
\mathbb{P}\left(\|Du(t,\xi)\|^2_{\mathcal{H}} < \veps\right) &\le \mathbb{P}\left(\|Du(t,\xi)\|^2_{\mathcal{H}_{\delta}} < \veps\right)\\
&\le \mathbb{P}\left(\|DG(t,\xi)\|^2_{\mathcal{H}_{\delta}} \ge \frac12\|DZ(t,\xi)\|_{\mathcal{H}_{\delta}}^2- \veps \right)\\
&\le \frac{\mathbb{E}\|DG(t,\xi)\|^{2 \tilde p}_{\mathcal{H}_{\delta}}}{\left( \frac12\|DZ(t,\xi)\|_{\mathcal{H}_{\delta}}^2- \veps\right)^{\tilde p}}\\
\end{align*}
Using the estimates obtained in Lemmas \ref{Du_s} and \ref{DF}, and choosing $\delta$ such that
$$\frac12 C_{T,x,\delta_0, \alpha} \delta^{2H-1}=2 \varepsilon,$$
for every $\varepsilon< \varepsilon_0=\frac14 C_{T,x,\delta_0, \alpha} \delta_0^{2H-1}$ we have 
\begin{equation}
\label{stima_P}
\mathbb{P}\left(\|Du(t,\xi)\|^2_{\mathcal{H}} < \veps\right) \le \frac{c \delta^{2 \tilde{p}}}{\veps^{\tilde{p}}}=C(\veps^{\frac{3-2H}{2H-1}})^{\tilde{p}}
\end{equation}

and since $\frac{3-2H}{2H-1}>0$ for any $H \in \left(\frac 12 ,1\right)$, the above estimate allows us to conclude the proof.


\end{proof}

\begin{remark}
Thanks to Definition \ref{def_D(t,x)} we can extend the result proved in Theorem \ref{exist_dens} to every $(t, x) \in \left[0,T\right] \times D$. This allows us to conclude that the image law of the solution of problem \eqref{eq:nl2} is absolutey continuous with respect to the Lebesgue measure on $\mathbb{R}$.
\end{remark}

\subsection{Smoothness of the density for the solution of the non homogeneous equation}
\label{sdsnhe}

In case the coefficients of the equation are more regular than just Lipschitz continuous, we can
prove the smoothness of the density of the random variable $u(t,x)$ that solves Eq.\eqref{eq:nl} for $(t,x)\in \left[0,T\right] \times D$,
analogously with the result in Lemma \ref{esis_dens_Z}. 


\begin{theorem}
Assume that the condition \textbf{(g2)} holds.
Then for each $(t,x)\in \left[0,T\right] \times D$, the random variable $u(t,x)$ has a density with respect to the Lebesgue measure
that is infinitely differentiable.
\end{theorem}

\begin{proof}
The idea is to apply Theorem \ref{theo:dens}. Let us at first consider $u(t, \xi)$ for $(t, \xi) \in \left[0,T\right] \times \partial D$. From Corollary \ref{cor} we have that  $u(t, \xi) \in \mathbb{D}^{\infty}$. It remains to prove that $\mathbb{E}(\|Du(t,\xi)\|)^{-p} < + \infty$ for every $p\ge 1$.
By Nualart\cite{Nualart2006}
Lemma 2.3.1, it suffices to prove that, for any $q \ge 2$, there exits $\veps_0(q)>0$ such that, for all $\veps < \veps_0$,
\begin{align*}
\mathbb{P}(\|Du(t,\xi)\|^2_{\mathcal{H}} < \veps) < \veps^q
\end{align*}
and this condition immediately follows from the estimate \eqref{stima_P} above,
choosing 
%
$\tilde{p}=\frac{q(2H-1)}{3-2H}$.

Notice that, as done for the above results, we can now extend this result to every $(t, x) \in \left[0,T \right] \times D$ and conclude that the density of the solution is infinitely differentiable with respect to the Lebesgue measure.
\end{proof}

\appendix
\section{Some supplementary lemmas}

The following lemmas are needed in the proof of Corollary \ref{co:stima}.

\begin{lemma}
\label{integral}
If $0 \le a \le d-1$, $0 \le b \le d-1$ then 
\begin{equation}
\int_{\partial D}\frac{{\rm d}y}{|x-y|^a |y-\xi|^b} \le 
\begin{cases}
c|x-\xi|^{d-1-a-b} \qquad \emph{if} \ a+b >d-1\\
c    \qquad \qquad \qquad \qquad \emph{if} \ a+b < d-1\\
\end{cases}
\end{equation}
\end{lemma}

\begin{lemma}
\label{stim_kernel_bordo}
Let $x \in \partial D$, $y \in \partial D$. The following estimate holds
\begin{equation}
\left| \frac{\partial}{\partial \nu(x)}\Gamma(t,x,s,y) \right| \le c(t-s)^{-\mu}|x-y|^{2 \mu-d}
\end{equation}
\end{lemma}

\begin{proof}
Proceeding as in the proof of \cite[Theorem 5.2.1]{Friedman1964}, we write
\begin{align*}
\frac{\partial}{\partial \nu(x)}\Gamma(t,x,s,y) &= \frac{\partial}{\partial \nu(x)}Z(t,x,s,y) \\
&=- \frac {(2 \pi)^{-\frac d2}}{2}(t-s)^{-\frac d2 -1} \exp\left(-\frac{|x-y|^2}{2(t-s)}\right) |x-y| \cos(N_{x}, \vec{xy}) 
\end{align*}
where with $N_x$ we denote the normal vector in $x$ and $\vec{xy}$ denotes the vector which connects $x$ with $y$. Since 
$$|\cos(N_n,\vec{xy})| \le C|x-y|,$$
we get
\begin{align*}
|\frac{\partial}{\partial \nu(x)}Z(t,x,s,y)| &\le C(t-s)^{-\frac d2 -1} \exp\left(-\frac{|x-y|^2}{2(t-s)}\right) |x-y|^2\\
&\le C(t-s)^{- \lambda}|x-y|^{-d+ 2 \lambda}
\end{align*}
where in the last inequality we have used the following analytic inequality:
for any $\alpha > 0$ and $x > 0$ it holds 
\begin{equation}\label{barbu:e1}
x^{\alpha} e^{-x} \le \alpha^{\alpha} e^{-\alpha}
\end{equation}
%
%
and so we obtain that for any $\mu > 0$ and $n \ge 1$ there exists a constant $K = K(n,\mu)$ such that for any $r > 0$:
\begin{equation}\label{c1}
t^{-n/2} e^{-k r^2/t} \le K t^{-\lambda} r^{-n+2\lambda}.
\end{equation}

\end{proof}

The following result is needed in order to prove the Malliavin regularity of the solution.

\begin{lemma}
\label{prel_lemma}
Under the Hypothesis \textbf{(a1')}, it holds 
\begin{equation}
\label{preliminary_lemma}
\int_S \left(\int_{\partial D}|x-\bar{y}|^{2 \mu-d}|\alpha(\sigma, \bar{y})| \, {\rm d}\bar{y} \right)^2 \mu({\rm d}\sigma) < C \|\alpha\|^2_{L^\theta(\partial D;L^2(S))}.
\end{equation}
\end{lemma}

\begin{proof}
Using the H\"{o}lder inequality with $\theta$ satisfying Hypothesis \textbf{(a1')}
, we get
\begin{multline*}
\int_S\left(\int_{\partial D}|\xi-\bar y|^{2 \mu-d}\alpha(\sigma, \bar{y}) {\rm d}\bar{y}\right)^2\mu({\rm d}\sigma)\\
\le\int_S \left(\int_{\partial D}|\xi- \bar{y}|^{\theta'(2 \mu-d)}{\rm d}\bar{y} \right)^{\frac {2}{\theta'}}\left(\int_{\partial D}|\alpha(\sigma, \bar{y})|^{\theta} {\rm d}\bar{y}\right)^{\frac {2}{\theta}}\mu({\rm d}\sigma)\\
\le C \int_S\left(\int_{\partial D}|\alpha(\sigma, \bar{y})|^{\theta} {\rm d}\bar{y}\right)^{\frac {2}{\theta}}\mu({\rm d}\sigma)
=C \|\alpha\|^2_{L^{\theta}(\partial D;L^2(S))}
\end{multline*}
Notice that the condition imposed on $\theta$ is needed in order to apply Lemma \ref{integral} and so have the boundedness of the integral
$$\int_{\partial D}|\xi- \bar{y}|^{\theta'(2 \mu-d)}{\rm d}\bar{y}.$$
\end{proof}

The following two lemmas are necessary for the proof of Theorem \ref{exist_dens}.

\begin{lemma}
\label{Du_s}
For every $\xi \in \partial D$ and $t \in [0,T]$
$$\mathbb{E}\|DG(t,\xi)\|^{p}_{\mathcal{H}_{\delta}} 
\le C_{L,p} \delta^{p} ,$$
for every $p \ge2$,
where 
\begin{align*}
G(t, \xi):=\int_0^t\int_{\partial D}p_N(t-s, \xi, \bar y)g(u(s,\bar y)){\rm d}\bar y{\rm d}s.
\end{align*}
\end{lemma}
\begin{proof}
We have that
\begin{align*}
\pmb{1}_{(t- \delta, t)}(s)D_{r, \sigma}G(t, \xi)= \pmb{1}_{(t- \delta, t)}(s)\int_0^t\int_{\partial D}p_N(t-s, \xi,\bar y)\partial_ug(u(s,\bar y))D_{r, \sigma}u(s,\bar y){\rm d}\bar y{\rm d}s
\end{align*}
and since $ D_{r, \sigma}u(t, \xi)=0$ for every $t<r$, we get
$$=\pmb{1}_{(t- \delta, t)}(s)\int_{t- \delta}^t\int_{\partial D}p_N(t-s, \xi,\bar y)\partial_ug(u(s,\bar y))D_{r, \sigma}u(s,\bar y){\rm d}\bar y{\rm d}s$$

Using H\"older inequality and the estimates on the kernel we get
\begin{multline*}
\mathbb{E}\|DG(t, \xi)\|^p_{\mathcal{H}_{\delta}} 
\le L^p\mathbb{E}\left(\int_{t-\delta}^t\int_{\partial D}p_N(t-s, \xi,\bar y) \|D_{r, \sigma}u(s,\bar y)\|_{\mathcal{H}} \, {\rm d}\bar y \, {\rm d}s \right)^p
\\
\le L^p \left(\int_{t- \delta}^t\int_{\partial D}p_N(t-s, \xi, \bar y)^q \, {\rm d}\bar y \, {\rm d}s \right)^{\frac pq}
\left( \int_{t- \delta}^t \int_{\partial D}\mathbb{E}\|D_{r, \sigma}u(s,\bar y)\|^p_{\mathcal{H}} \, {\rm d}\bar y \, {\rm d}s \right)
\end{multline*}
The term $\mathbb{E}\|D_{r, \sigma}u(s,\bar y)\|^p_{\mathcal{H}}$ is finite thanks to Theorem \ref{D12_sol},
so we obtain
\begin{align*}
\mathbb{E}\|DG(t, \xi)\|^p_{\mathcal{H}_{\delta}} \le  L^p \left(\int_{t- \delta}^t\int_{\partial D}p_N(t-s, \xi, \bar y){\rm d}\bar y\,{\rm d}s \right)^{p},
\end{align*}
and using \eqref{stima_ok} it follows that
\begin{align*}
\left(\int_{t- \delta}^t\int_{\partial D}p_N(t-s, \xi, \bar y)^q{\rm d}\bar y\,{\rm d}s \right)^{p}
&\le \left( \int_{\partial D}|\xi-\bar y|^{(2 \mu-d)}{\rm d}\bar y\right)^{p} \left( \int_{t- \delta}^t(t-s)^{- \mu }{\rm d}s\right)^{p}\\
&\le C \delta^{p(1-\mu)} \le C \delta^p,
\end{align*}
and the proof is complete.
\end{proof}

Next lemma is concerned with the norm of the Malliavin derivative $D Z(t,x)$ of the stochastic convolution process $Z(t,x)$ in the space $\cH_\delta$ and
is a refinement of the results in Lemma \ref{esis_dens_Z}.

\begin{lemma}
\label{DF}
Given $\delta_0>0$, for every $\delta < \delta_0$ and every $x \in D$ and $t \in ]0,T]$,
\begin{align*}
\|DZ (t,x)\|^2_{\mathcal{H}_{\delta}} \ge C_{\delta_0,T, \alpha_0} \delta^{2H-1}.
\end{align*}
\end{lemma}

\begin{proof}

Proceeding as in the proof of Lemma \ref{esis_dens_Z}, we get

\begin{align*}
\|D \, Z(t,x)\|^2_{\cH_{\delta}} =& \int_{t-\delta}^t \int_{t-\delta}^t \int_S D_{s,\sigma}Z(t,x) \, D_{r,\sigma}Z(t,x) \, |s-r|^{2H-2} \, \mu({\rm d}\sigma) \, {\rm d}r \, {\rm d}s
\\
&\ge C \delta^{d-1}\int_{t-\delta}^t \int_{t-\delta}^t(t-s)^{-\frac d2}(t-r)^{-\frac d2}|s-r|^{2H-2}\exp\left({-\frac{\delta^2}{t-s}}\right)\exp\left({-\frac{\delta^2}{t-r}}\right) \, {\rm d}r \, {\rm d}s
\\
&= C \delta^{d-1}\int_0^{\delta} \int_0^{\delta}\sigma^{-\frac d2}\rho^{-\frac d2}|\rho-\sigma|^{2H-2}\exp\left({-\frac{\delta^2}{\sigma}}\right)\exp\left({-\frac{\delta^2}{\rho}}\right) \, {\rm d}\sigma \, {\rm d}\rho
\\
&\ge C \delta^{d-1}e^{-2\delta}\int_0^{\delta} \int_0^{\delta}\sigma^{-\frac d2}\rho^{-\frac d2}|\rho-\sigma|^{2H-2} \, {\rm d}\sigma \, {\rm d}\rho
\\
&\ge C \delta^{-1}e^{-2\delta}\int_0^{\delta} \int_0^{\delta}|\rho-\sigma|^{2H-2} \, {\rm d}\sigma \, {\rm d}\rho
\ge C \delta^{-1}e^{-2\delta} \frac{\delta^{2H}}{2H(2H-1)} = C_{\delta_0,T, \alpha_0} \delta^{2H-1}.
\end{align*}
\end{proof}

\begin{lemma}\label{l-dopo}
Let $\{V_{n},\ n \ge 0\}$ be a sequence of processes in $L^p_W(0,T \times \partial D; \cH^{\otimes k})$, $p \ge 2$, {{$k\ge 0$ (in the case $k=0$ we mean $\mathcal{H}=\mathbb{R}$)}} defined recursively by
\begin{align*}
V_{n+1}(t,x) = V_0(t,x)+\int_r^t\int_{\partial D}p_N(t-s, x, \bar y) h_n(s, \bar y)V_n(s,\bar y) \, {\rm d} \bar y \, {\rm d}s
\end{align*}
and $V_0$ is given. Moreover, $h_n$ is a sequence of equi-bounded real-valued random processes.

Then there exists a suitable $\lambda$ in the definition of the norm $\|\cdot\|_{L^p_W}$ (see formula \eqref{e:0207-1}) 
and a constant $C < 1$ such that the following
estimate holds (compare for instance with \eqref{contraction}):
\begin{align}
\label{e:0207-2}
\| V_{n+1} - V_n\|_{L^p_W} \le C \, \|V_n - V_{n-1}\|_{L^p_W}.
\end{align}
\end{lemma}

\begin{proof}
In order to prove \eqref{e:0207-2}, let us consider in the space $L^p_W(\left[0,T\right] \times \partial D; \mathcal{H}^{\otimes k})$ the equivalent norm
\begin{align*}
\|f\|_{L^p_{W}}= \mathbb{E}\int_0^Te^{-\lambda t}\|f(t)\|^p_{L^p(\partial D; \mathcal{H}^{\otimes k})} \, {\rm d}t,
\end{align*}
for some constant $\lambda$ to be chosen later. The norm is equivalent to the standard one.

Setting for simplicity $\tilde V(t, x):=(V_n-V_{n-1})(t, x)$ and using H\"older inequality and the assumption of equi-boundedness on $h$,
$\sup h(t,\bar x) \le L < \infty$ $\bP$-a.s.,
we have that 
\begin{align*}
\mathbb{E}\|(V_{n+1}-V_n)(t, x)\|^p_{\mathcal{H}^{\otimes k}}
&=\mathbb{E}\left| \left| \int_r^t\int_{\partial D}p_N(t-s, x,\bar y)h(s,\bar y))(V_n-V_{n-1})(s,\bar y) \, {\rm d}\bar y \, {\rm d}s\right|\right|^p_{\mathcal{H}^{\otimes k}}
\\
&\le L \, \mathbb{E}\left( \int_r^t\int_{\partial D}p_N(t-s, x,\bar y)\|\tilde{V}(s,\bar y)\|_{\mathcal{H}^{\otimes k}} \, {\rm d}\bar y \, {\rm d}s\right)^p
\\
&\le L \, \left(\int_0^t \int_{\partial D} p_N(t-s, x,\bar y)e^{\lambda s(q-1)} \, {\rm d}\bar y \, {\rm d}s\right)^{\frac pq}\times 
\\
& \qquad\qquad \times\mathbb{E} \left(\int_0^t\int_{\partial D}p_N(t-s, x,\bar y)e^{-\lambda s}\|\tilde{V}(s,\bar y)\|_{\mathcal{H}^{\otimes k}} \, {\rm d}\bar y \, {\rm d}s \right).
\end{align*}
Using estimate \eqref{stima_ok} we obtain the following 
\begin{align*}
\left(\int_0^t\int_{\partial D}p_N(t-s, x,\bar y)e^{\lambda s(q-1)}{\rm d}\bar y{\rm d}s\right)^{\frac pq}&= \left( \int_0^t \int_{\partial D}p_N(s, x, \bar y)e^{\frac{\lambda(t-s)}{p-1}} \, {\rm d }\bar y \, {\rm d}s\right)^{p-1} 
\\
&\le \left(\int_{\partial D}|x-\bar y|^{2 \mu-d}{\rm d}\bar y \int_0^t s^{-\mu}e^{\frac{\lambda (t-s)}{p-1}} \, {\rm d}s \right)^{p-1}
\\
&\le 
C \frac{1}{\lambda^{(1-\mu)/(p-1)}}
e^{\lambda t}
\end{align*}
where $C = C(p,\mu)$ is uniformly bounded for fixed $\mu$ and $p \ge 2$.
So, by means of the Fubini's Theorem, we get
\begin{align*}
\|V_{n+1}-V_n\|&^p_{L^p_W}= \int_r^T\int_{\partial D}\mathbb{E}\|V_{n+1}-V_n\|^p_{\mathcal{H}^{\otimes k}}\,{\rm d}x \,e^{-\lambda t}\,{\rm d}t \\
&\le C\frac{1}{\lambda^{(1-\mu)/(p-1)}}\int_r^T\int_{\partial D} e^{\lambda t} \int_r^t\int_{\partial D}p_N(t-s, x, \bar y)e^{-\lambda s}\,\mathbb{E}\|\tilde V(s,\bar y)\|^p_{\mathcal{H}^{\otimes k}}\,{\rm d}\bar y\, {\rm d}s \,{\rm d}x \,e^{-\lambda t}\,{\rm d}t\\
&=C\frac{1}{\lambda^{(1-\mu)/(p-1)}}\int_r^T \int_{\partial D}\mathbb{E}\|\tilde{V}(s, \bar y)\|^p_{\mathcal{H}^{\otimes k}}e^{-\lambda s}\left(\int_s^T\int_{\partial D}p_N(t-s, \xi,\bar y)\,{\rm d}x\,{ \rm d}t \right){\rm d}\bar y \, {\rm d}s
\\
&\le C\frac{1}{\lambda^{(1-\mu)/(p-1)}}\int_r^T \int_{\partial D}\mathbb{E}\|\tilde V(s, \bar y)\|^p_{\mathcal{H}^{\otimes k}}e^{-\lambda s}\,{\rm d}\bar y \,{\rm d}s
\\
&\le C \lambda^{\frac{\mu-1}{p-1}}\|V_{n}-V_{n-1}\|^p_{L^p_W},
\end{align*}
where, thanks again to Lemma \ref{integral} (recall that $\frac 12 < \mu < 1$ and so $2 \mu-d < d-1$),
\begin{align}
\label{C_T}
\int_s^T\int_{\partial D}p_N(t-s,\xi, \bar{y}){\rm d}\xi {\rm d}t &\le \int_s^T\int_{\partial D}|\xi-\bar{y}|^{2 \mu -d}(t-s)^{-\mu} {\rm d}\xi{\rm d}t\notag \\ 
&\le \int_0^Tt^{-\mu}{\rm d}t \int_{\partial D}|\xi-\bar{y}|^{2\mu-d} {\rm d}\xi \le C_{\mu}T^{1-\mu} =C_{\mu,T}.
\end{align}

Then we see that there exists $\lambda$ large
 enough such that $C\lambda^{\frac{\mu-1}{p-1}}\le 1-\varepsilon <1$ and this proves the Lemma.
\end{proof}


\end{document}